\documentclass[11pt,reqno,sumlimits]{amsart}
\usepackage{amsfonts, amsmath, amscd, amssymb, euscript, amsthm, array, booktabs,
dcolumn, shortvrb, tabularx, units, url}
\usepackage[pdftex]{graphicx}
\usepackage[all]{xy}
\normalsize

\DeclareMathOperator{\rk}{rank}
\DeclareMathOperator{\fl}{flat}
\DeclareMathOperator{\vect}{Vect}
\DeclareMathOperator{\pt}{\ast}

\DeclareMathOperator{\der}{deR}

\DeclareMathOperator{\LL}{L}

\DeclareMathOperator{\FL}{FL}

\DeclareMathOperator{\exact}{exact}

\DeclareMathOperator{\dr}{dR}

\DeclareMathOperator{\odd}{odd}
\DeclareMathOperator{\even}{even}
\DeclareMathOperator{\im}{Im}

\DeclareMathOperator{\ch}{ch}

\DeclareMathOperator{\U}{U}
\DeclareMathOperator{\BU}{BU}

\DeclareMathOperator{\CS}{CS}

\begin{document}
\theoremstyle{definition}
\setlength{\baselineskip}{1.0\baselineskip}
\newtheorem{defi}{Definition}
\newtheorem{remark}{Remark}
\newtheorem{coro}{Corollary}
\newtheorem{exam}{Example}
\newtheorem{lemma}{Lemma}
\newtheorem{thm}{Theorem}
\newtheorem{prop}{Proposition}
\newcommand{\wt}[1]{{\widetilde{#1}}}
\newcommand{\ov}[1]{{\overline{#1}}}
\newcommand{\wh}[1]{{\widehat{#1}}}
\newcommand{\poin}{Poincar$\acute{\textrm{e }}$}
\newcommand{\deff}[1]{{\bf\emph{#1}}}
\newcommand{\boo}[1]{\boldsymbol{#1}}
\newcommand{\abs}[1]{\lvert#1\rvert}
\newcommand{\norm}[1]{\lVert#1\rVert}
\newcommand{\inner}[1]{\langle#1\rangle}
\newcommand{\poisson}[1]{\{#1\}}
\newcommand{\biginner}[1]{\Big\langle#1\Big\rangle}
\newcommand{\set}[1]{\{#1\}}
\newcommand{\Bigset}[1]{\Big\{#1\Big\}}
\newcommand{\BBigset}[1]{\bigg\{#1\bigg\}}
\newcommand{\dis}[1]{$\displaystyle#1$}
\newcommand{\R}{\mathbb{R}}
\newcommand{\N}{\mathbb{N}}
\newcommand{\Z}{\mathbb{Z}}
\newcommand{\Q}{\mathbb{Q}}
\newcommand{\E}{\mathcal{E}}
\newcommand{\G}{\mathcal{G}}
\newcommand{\F}{\mathcal{F}}
\newcommand{\V}{\mathcal{V}}
\newcommand{\W}{\mathcal{W}}
\newcommand{\SSS}{\mathcal{S}}
\newcommand{\h}{\mathbb{H}}
\newcommand{\g}{\mathfrak{g}}
\newcommand{\C}{\mathbb{C}}
\newcommand{\A}{\mathcal{A}}
\newcommand{\M}{\mathcal{M}}
\newcommand{\HH}{\mathcal{H}}
\newcommand{\D}{\mathcal{D}}
\newcommand{\PP}{\mathcal{P}}
\newcommand{\K}{\mathcal{K}}
\newcommand{\RR}{\mathcal{R}}
\newcommand{\RRR}{\mathscr{R}}
\newcommand{\DDD}{\mathscr{D}}
\newcommand{\so}{\mathfrak{so}}
\newcommand{\gl}{\mathfrak{gl}}
\newcommand{\aaa}{\mathbb{A}}
\newcommand{\bbb}{\mathbb{B}}
\newcommand{\DD}{\mathsf{D}}
\newcommand{\ccc}{\bold{c}}
\newcommand{\sss}{\mathbb{S}}
\newcommand{\cdd}[1]{\[\begin{CD}#1\end{CD}\]}
\normalsize
\title{Refined hexagons for differential cohomology}
\author{Man-Ho Ho}
\address{Department of Mathematics\\ Hong Kong Baptist University}
\email{homanho@math.hkbu.edu.hk}
\maketitle
\nocite{*}
\begin{abstract}
Differential characters and differential $K$-theory are refinements of ordinary cohomology
theory and topological $K$-theory respectively, and they are examples of differential
cohomology theory. Each of these differential cohomology theories fits into a hexagon on 
the cohomology level. We show that these differential cohomology theories fit into hexagons 
on the cocycle level, and the hexagons on the cocycle level induce the hexagons on the
cohomology level.
\end{abstract}
\tableofcontents
\section{Introduction}

Given a generalized cohomology theory $E$ its differential extension $\wh{E}$, also known
as differential cohomology \cite{BS10a}, is a refinement of $E$ to the category of smooth
manifolds by incorporating differential form information. Differential character $\wh{H}$
\cite{CS85} and differential $K$-theory $\wh{K}$ \cite{HS05, BS09, FL10, SS10} are examples
of differential extensions of ordinary cohomology theory $H$ and topological $K$-theory
$K$ respectively. $\wh{H}$ and $\wh{K}$ have been studied extensively in recent years due
to its importance in geometry and topology \cite{BS09, FL10, GV09} and its applications
in theoretical physics \cite{F00, HS05}.

Let $X$ be a smooth manifold and $A$ a proper subring of $\R$. Differential characters
$\wh{H}^k(X; \R/A)$ fits into the following hexagon, i.e., the diagonal sequences are
exact and every square and triangle commutes \cite{CS85} (the details is given in Example
\ref{exam 1}).
\begin{equation}\label{eq 1}
\xymatrix{\scriptstyle 0 \ar[dr] & \scriptstyle & \scriptstyle & \scriptstyle &
\scriptstyle 0 \\ & \scriptstyle H^{k-1}(X; \R/A) \ar[rr]^{-\beta} \ar[dr]^i & \scriptstyle
& \scriptstyle H^k(X; A) \ar[ur] \ar[dr]^{\ch} & \scriptstyle \\ \scriptstyle H^{k-1}
(X; \R) \ar[ur]^{\alpha} \ar[dr]_s & \scriptstyle & \scriptstyle\wh{H}^k(X; \R/A)
\ar[ur]^I \ar[dr]^R & \scriptstyle & \scriptstyle H^k(X; \R) \\ \scriptstyle &
\scriptstyle \frac{\Omega^{k-1}(X)}{\Omega^{k-1}_{A}(X)} \ar[rr]_{d} \ar[ur]^a &
\scriptstyle & \scriptstyle\Omega^k_{A}(X) \ar[ur]_{\der} \ar[dr] & \scriptstyle \\
\scriptstyle 0 \ar[ur] & \scriptstyle & \scriptstyle & \scriptstyle & \scriptstyle 0}
\end{equation}
Moreover, this hexagon uniquely characterizes differential characters 
\cite[Theorem 1.1]{SS08}, i.e., for any two functors from the category of smooth manifolds 
to the category of abelian groups fitting into the hexagon (\ref{eq 1}) there exists a 
unique natural equivalence between the two functors. Thus one can define differential 
characters $\wh{H}$ to be a functor fitting into the hexagon (\ref{eq 1}).

In \cite{FL10, SS10} it is proved that differential $K$-theory $\wh{K}$ fits into an
analogous hexagon.
\begin{equation}\label{eq 2}
\xymatrix{\scriptstyle 0 \ar[dr] & \scriptstyle & \scriptstyle & \scriptstyle &
\scriptstyle 0 \\ & \scriptstyle K^{-1}_{\LL}(X; \R/\Z) \ar[rr]^B \ar[dr]^i &
\scriptstyle & \scriptstyle K(X) \ar[ur] \ar[dr]^{\ch} & \scriptstyle \\ \scriptstyle
H^{\odd}(X; \R) \ar[ur]^{\alpha} \ar[dr]_s & \scriptstyle & \scriptstyle\wh{K}_{\FL}(X)
\ar[ur]^I \ar[dr]^R & \scriptstyle & \scriptstyle H^{\even}(X; \R) \\ \scriptstyle &
\scriptstyle \frac{\Omega^{\odd}(X)}{\Omega^{\odd}_{\BU}(X)} \ar[rr]_d \ar[ur]^a &
\scriptstyle & \scriptstyle \Omega^{\even}_{\BU}(X) \ar[ur]_{\der} \ar[dr] &
\scriptstyle \\ \scriptstyle 0 \ar[ur] & \scriptstyle & \scriptstyle & \scriptstyle
& \scriptstyle 0}
\end{equation}
where $K^{-1}_{\LL}(X; \R/\Z)$ is the flat $K$-theory given in \cite{L94}.

The existence and uniqueness of differential extension of generalized cohomology theory
is studied in \cite{BS10a}, where axioms are proposed for defining it. Moreover, the 
existence and the uniqueness of the differential extension are proved under some natural 
assumptions \cite[Theorem 3.10]{BS10a}.

In \cite[Section 3]{HS05} differential cohomology theory is defined to be the cohomology 
of a cochain complex by using homotopic-theoretic method. For differential characters and 
differential $K$-theory there are concrete descriptions of the cochain complexes defining
them in terms of ordinary cohomology and $K$-theory respectively. Thus a natural question 
is whether the hexagons (\ref{eq 1}) and (\ref{eq 2}) hold on the cocycle level in a 
certain sense. We show that these are true. More precisely, we construct hexagons for 
differential characters and differential $K$-theory on the cocycle level (see (\ref{eq 5}) 
and (\ref{eq 6})). We prove that the all the triangles and squares in these hexagons 
commute, certain sequences are exact and the all maps involved are natural, so these 
hexagons are diagrams of functors and natural transformations. It is obvious to see that 
when these hexagons are passed to the cohomology level they coincide with (\ref{eq 1}) 
and (\ref{eq 2}) respectively. Note  that the statements for the hexagons on the cocycle 
level are (expected to be) slightly weaker than the ones on the cohomology level as there
are more non-zero elements in the group of differential cocycles.

Recall from \cite[p. 341]{HS05} that a differential character $x\in\wh{H}^k(X; \R/A)$ is 
said to admit a topological trivialization if there exists $s\in\wh{C}^{k-1}(X; \R/A)$ 
such that $x=\wh{d}s$, or equivalently $x\in\wh{H}^k(X; \R/A)$ admits a topological 
trivialization if and only if $I(x)=0$ \cite[\S 3]{R14}. In our language, 
\cite[Proposition 3.12, Theorem 3.15]{R14} says when $x\in\wh{H}^k(X; \R/\Z)$ admits a 
topological trivialization, the right hand square of (\ref{eq 5}) with $\R/\Z$ 
coefficients commutes, and is a torsor for the right-hand square of (\ref{eq 1}). Thus 
Theorem \ref{thm 1} can be regarded as a generalization of 
\cite[Proposition 3.12, Theorem 3.15]{R14} to differential characters not necessarily 
admitting topological trivializations and to all coefficients.

The paper is organized as follow. In Section 2 we present the necessary background 
materials, including Bunke-Schick axioms of differential extension of generalized 
cohomology theory, the hexagons for differential characters and differential $K$-theory, 
and Hopkins-Singer's differential cocycle model. The main results are proved in Section 3.

\section*{Acknowledgement}

The author would like to thank David Fried as the idea of the paper came up from his very
inspiring comments on the existence of a ``general" hexagon, and to an anonymous person 
for suggesting the use of mapping cone.

\section{Background materials}

We recall the axiomatic definition of differential cohomology given in \cite{BS10a}.
\begin{defi}\cite[Definition 1.1]{BS10a}\label{defi 1.1}
Let $X$ be a smooth manifold, $E$ a generalized cohomology theory and
$\ch:E^\bullet(X)\to H^\bullet(X; V)$ a natural transformation of cohomology theories,
where $V:=E^\bullet(\pt)\otimes_\Z\R$ is a $\Z$-graded real vector space. A
\deff{differential extension} of $(E, \ch)$ is a quadruple $(\wh{E}, R, I, a)$ such
that
\begin{enumerate}
  \item $\wh{E}$ is a contravariant functor from the category of smooth manifolds
        to $\Z$-graded abelian groups,
  \item $R$, $a$ and $I$ are natural transformations of $\Z$-graded abelian
        group-valued functors such that $R\circ a=d$.
\end{enumerate}
Moreover, the following diagram commutes
\begin{equation}\label{eq 3}
\begin{CD}
\wh{E}^\bullet(X) @>I>> E^\bullet(X) \\ @VRVV @VV\ch V \\ \Omega_{d=0}(X; V) @>>\der>
H^\bullet(X; V)
\end{CD}
\end{equation}
where $\Omega_{d=0}(X; V)$ is the group of closed forms on $X$ with values in $V$,
and the following sequence is exact
\begin{equation}\label{eq 4}
\begin{CD}
E^{\bullet-1}(X) @>\ch>> \displaystyle\frac{\Omega^{\bullet-1}(X; V)}{\im(d)} @>a>>
\wh{E}^\bullet(X) @>I>> E^\bullet(X) @>>> 0
\end{CD}
\end{equation}
\end{defi}
The following examples summarize differential characters and Freed-Lott differential
$K$-theory.
\begin{exam}\label{exam 1}
Let $X$ be a smooth manifold and $A$ a proper subring of $\R$. If $E=H$, differential
characters $\wh{H}(X; \R/A)$ is the differential extension of $H(X; A)$. A differential
character of degree $k\in\N$ is a homomorphism $f:Z_{k-1}(X)\to\R/A$ with a differential
form $\omega\in\Omega^k(X)$ such that
$$f(\partial c)=\int_c\omega\mod A,$$
where $c\in C_k(X)$. It can be shown that $\omega\in\Omega^{k-1}_A(X)$, where 
$\Omega^{k-1}_A(X)$ is the group of closed $(k-1)$-forms with periods in $A$. In 
(\ref{eq 1}), the maps are
\begin{displaymath}
\begin{split}
R(f)&=\omega,\\
a(\eta)(z)&=\int_z\eta\mod A,\\
I(f)&=c_f,\qquad\textrm{ which satisfies } j_A(c_f)=[\omega],\\
i([u])&=u|_{Z_{k-1}(X)},\\
\der(\omega)([c])&=\int_{[c]}[\omega],\qquad\textrm{ where } [c]\in H_k(X; \R),\\
s([u])&=u \textrm{ which is a representative in } \der^{-1}([u]),\\
\ch([u])&=j_A([u]),
\end{split}
\end{displaymath}
where $j_A:A\hookrightarrow\R$ is the inclusion of coefficients.
\end{exam}
\begin{exam}\label{exam 2}
If $E=K$, differential $K$-theory $\wh{K}(X)$ is the differential extension of topological
$K$-theory. In this paper we use Freed-Lott differential $K$-theory. Consider an abelian
monoid $\wh{\vect}(X)$ consists of elements of the form $\E=(E, h, \nabla, \phi)$, where
$E\to X$ is a Hermitian bundle, $h$ the Hermitian metric, $\nabla$ a connection compatible
with $h$ and \dis{\phi\in\frac{\Omega^{\odd}(X)}{\Omega^{\odd}_{\exact}(X)}}. $\E_1=\E_2$
if $E_1\cong E_2$ and $\phi^E-\phi^F=\CS(\nabla^F, \nabla^E)\mod\Omega^{\odd}_{\exact}(X)$.
The only relation is $\E_1\sim\E_2$ if and only if there exists $\F=(F, h^F, \nabla^F,
\phi^F)\in\wh{\vect}(X)$ such that $\E_1\oplus\F=\E_2\oplus\F$, i.e., $E_1\oplus F\cong
E_2\oplus F$ and $\phi_1-\phi_2=\CS(\nabla^{E_2}\oplus\nabla^F, \nabla^{E_1}\oplus\nabla^F)
\mod\Omega^{\odd}_{\exact}(X)$. Define \dis{\wh{K}_{\FL}(X)=\wh{\vect}(X)/\sim}. Note that,
strictly speaking the $H^{\odd}(X; \R)$ in the left most entry of (\ref{eq 2}) had been 
identified with $H^{\odd}_{\dr}(X)$. In (\ref{eq 2}), the maps are
\begin{displaymath}
\begin{split}
I(E, h, \nabla, \phi)&=[E],\\
R(E, h, \nabla, \phi)&=\ch(\nabla)+d\phi,\\
a(\phi)&=(0, 0, d, \phi),\\
\alpha(\omega)&=([\C^n], \nabla^{\fl}, \omega)-([\C^n], \nabla^{\fl}, 0),\\
B(\E-\F)&=[E]-[F],
\end{split}
\end{displaymath}
where $i:K^{-1}_{\LL}(X; \R/\Z)\to\wh{K}(X)$ is the canonical inclusion, and
$$\Omega^{\bullet}_{\BU}(X):=\set{\omega\in\Omega^{\bullet}_{d=0}(X)|[\omega]\in\im
(\ch^\bullet:K^{-(\bullet\mod 2)}\to H^\bullet(X; \Q))},$$
where $\bullet\in\set{\even, \odd}$. The maps $s$ and $\der$ are defined in a similar way 
as in Example \ref{exam 1}, and $\ch$ is the Chern character composed with $j_\Q$.
\end{exam}

We recall Hopkins-Singer's differential cocycle model of differential character
\cite[Section 3.2]{HS05}.
\begin{defi}
Fix $q\in\N$. Define a complex $\wh{C}(q)^\ast(X; A)$ by
$$\wh{C}(q)^k(X; A)=\left\{\begin{array}{ll} C^k(X; A)\times C^{k-1}(X; \R)\times\Omega^k
(X), & \textrm{ for } k\geq q \\ C^k(X; A)\times C^{k-1}(X; \R), & \textrm{ for } k
<q \end{array}\right..$$
with differential $\wh{d}:\wh{C}(q)^k(X; A)\to\wh{C}(q)^{k+1}(X; A)$ by
\begin{displaymath}
\begin{split}
\wh{d}(c, T, \omega)&=(\delta c, \der(\omega)-j_A(c)-\delta T, d\omega)\qquad \textrm{ for }
k\geq q\\
\wh{d}(c, T)&=\left\{\begin{array}{ll} (\delta c, -j_A(c)-\delta T, 0), & \textrm{ for }
k=q-1\\ (\delta c, -j_A(c)-\delta T), & \textrm{ for } k<q-1 \end{array}\right.,
\end{split}
\end{displaymath}
where \dis{\der(\omega)(c)=\int_c\omega} for $c\in C_k(X)$. One can easily check that
$\wh{d}^2=0$. $(\wh{C}(q)^\ast(X; A), \wh{d})$ is the differential cocycle model of
differential characters.
\end{defi}
Let
$$\wh{Z}(k)^k(X; A):=\ker(\wh{d}:\wh{C}(k)^k(X; A)\to\wh{C}(k)^{k+1}(X; A))$$
be the groups of cocycles and $\wh{B}(k)^k(X; A)$ the subgroup of coboundaries. Its 
cohomology group is denoted by \dis{\wh{H}(k)^k(X; A)}. In \cite[Section 3.2]{HS05} it 
is proved that the map $f:\wh{H}(k)^k(X; A)\to\wh{H}^k(X; \R/A)$ given by
\begin{equation}\label{eq 4.5}
f([c, T, \omega])=T|_{Z_{k-1}(X)}\mod A
\end{equation}
is an isomorphism of graded commutative rings.

In section \ref{s 3.1} we will replace $H^\bullet(X; \R/A)$ by the cohomology of the
cochain complex of the mapping cone induced by the inclusion of coefficients 
$j_A:A\hookrightarrow\R$. We briefly recall the construction and in particular the maps 
involved (for details, see \cite{W94}). Define the cochain complex $\set{C(X; j_A), 
\delta_{j_A}}$, where
$$C(X; j_A)^k=C^{k+1}(X; A)\oplus C^k(X; \R),\qquad\delta_{j_A}(u, v)=(-\delta u, \delta
v-j_A(u)).$$
Denote by $Z(X; j_A)^k$ and $B(X; j_A)^k$ the kernel of $\delta_{j_A}^k$ and the image 
of $\delta_{j_A}^{k-1}$ respectively. It is well known that the cohomology of cochain
complex $\set{C(X; j_A), \delta_{j_A}}$ satisfies the following short exact sequence
\cdd{0 @>>> C^k(X; \R) @>\alpha>> C^k(X; j_A) @>\gamma>> C^{k+1}(X; \Q) @>>> 0}
where $\alpha(c)=(0, c)$ and $\gamma(b, c)=-b$. Henceforth we will write all the induced
maps on the cochain level and the cohomology level by the same symbol. The induced long 
exact sequence is given by
\cdd{\scriptstyle\cdots @>>> \scriptstyle H^k(X; \Q) @>j_A>> \scriptstyle H^n(X; \R)
@>\alpha>> \scriptstyle H^k(C(X; j_A)) @>\gamma>> \scriptstyle H^{k+1}(X; \Q) @>>>
\scriptstyle\cdots}
By the Five lemma, $H^\bullet(C(X; j_A))$ is isomorphic to $H^\bullet(X; \R/A)$.

\section{Main results}

In this section we construct hexagons for differential characters and differential
$K$-theory on the cocycle level and prove that each of these hexagons induces the
corresponding hexagon on the cohomology level. We take $H$ to be the singular cohomology. 
See Remark \ref{remark 1} for the necessary modification of the proof of Theorem 
\ref{thm 1} if $H$ is taken to be the simplicial cohomology.

\subsection{Refined hexagon for differential characters}\label{s 3.1}

We refine the hexagon for differential characters $\wh{H}(X; \R/A)$ with coefficients
in any fixed proper subring $A$ of $\R$ to the cocycle level.
\begin{thm}\label{thm 1}
Let $X$ be a smooth manifold. Consider the following diagram
\begin{equation}\label{eq 5}
\xymatrix{\scriptstyle 0 \ar[dr] & \scriptstyle & \scriptstyle & \scriptstyle &
\scriptstyle 0 \\
\scriptstyle & \scriptstyle Z(X; j_A)^{k-1} \ar[dr]_i \ar[rr]^\beta & \scriptstyle &
\scriptstyle Z^k(X; A)\times B^k(X; \R) \ar[ur] \ar[dr]^{\ch} & \scriptstyle \\
\scriptstyle \Omega^{k-1}_{d=0}(X) \ar[ur]^b \ar[dr]_\iota & \scriptstyle & \scriptstyle
\wh{Z}(k)^k(X; A) \ar[ur]^I \ar[dr]_R & \scriptstyle & \scriptstyle Z^k(X; \R) &
\scriptstyle \\ \scriptstyle & \scriptstyle \Omega^{k-1}(X) \ar[ur]^a \ar[rr]_d &
\scriptstyle & \scriptstyle\Omega^k_A(X) \ar[ur]_{\der} \ar[dr] & \scriptstyle \\
\scriptstyle 0 \ar[ur] & \scriptstyle & \scriptstyle & \scriptstyle & \scriptstyle 0 }
\end{equation}
where the maps will be defined in the proof. Every square and triangle commutes, the main
diagonal sequence is exact, the map $a$ is injective and the map $I$ is surjective.
Moreover, $(\ref{eq 5})$ induces $(\ref{eq 1})$, i.e., the commutativity of the squares
and triangles and the exactness of the diagonal sequences of (\ref{eq 1}) can be deduced
from (\ref{eq 5}).
\end{thm}
Note that the off diagonal sequence is not exact, however.
\begin{proof}
The maps are defined as follows:
\begin{displaymath}
\begin{split}
I(c, T, \omega)&=(c, \delta T),\\
R(c, T, \omega)&=\omega,\\
\der(\omega)(c)&=\int_c\omega,\qquad\textrm{ where } c\in C_k(X),\\
a(\eta)&=(0, \der(\eta), d\eta),\\
\ch(u, \delta S)&=j_A(u)+\delta S,\\
\beta(u, v)&=(-u, \delta v),\\
b(\omega)&=(0, \der(\omega)),\\
\iota(\omega)&=\omega,\\
i(u, v)&=(-u, v, 0).
\end{split}
\end{displaymath}

The injectivity of the maps $a$ and $i$ are obvious.

For the surjectivity of $R$, let $\omega\in\Omega^k_A(X)$. Then there exists a unique
$[c]\in H^k(X; A)$ such that $\der([\omega])=j_A[c]$. Thus $\der(\omega)=j_A(c)+\delta
T\in Z^k(X; \R)$ for some $T\in C^{k-1}(X; \R)$. Note that $(c, T, \omega)\in\wh{Z}(k)^k
(X; A)$ and $R(c, T, \omega)=\omega$.

For the surjectivity of $I$, let $(c, \delta T)\in Z^k(X; A)\times B^k(X; \R)$. Since
there exists a unique de Rham class $[\omega]\in H^k_{\dr}(X)$ such that $[j_A(c)+\delta
T]=\der([\omega])$, i.e., there exists $T'\in C^{k-1}(X; \R)$ such that
$$\der(\omega)-j_A(c)-\delta T=\delta T'.$$
Choose a representative $[d\eta]$ of $0=\der^{-1}([\delta T'])$, where $\eta\in
\Omega^{k-1}(X)$. We may assume $\der(d\eta)=\delta T'$, and therefore
$$\der(\omega-d\eta)-j_A(c)=\delta T.$$
It follows that $(c, T, \omega-d\eta)\in\wh{Z}(k)^k(X; A)$ and $I(c, T, \omega-d\eta)=
(c, \delta T)$.

To prove the exactness of the main diagonal sequence, note that $R\circ i=0$ follows
directly from the definitions of $R$ and $i$. Let $(c, T, \omega)\in\ker(R)$. Then
$\omega=0$ and therefore $\delta T+j_A(c)=0$. Thus $(-c, T)\in Z(X; j_A)^{k-1}$ and 
$i(-c, T)=(c, T, 0)$. Therefore $\im(i)=\ker(R)$.

For the commutativity of the right square, since $(c, T, \omega)\in\wh{Z}(k)^k(X; A)$
implies $\delta T=\omega-c$, we have
$$\ch(I(c, T, \omega))=\ch(c, \delta T)=r(c)+\delta T=\der(R(\omega)).$$

Note that the commutativity of the upper and the lower triangles; i.e., $R\circ a=d$ and
$I\circ i=\beta$, follows directly from definitions of $R$, $a$, $i$ and $\beta$.
Similarly, the commutativity of the left square; i.e., $i\circ b=a\circ\iota$, follows
directly from the definitions of the maps involved.

We prove that (\ref{eq 5}) induces (\ref{eq 1}), i.e., when $\wh{Z}(k)^k(X; A)$ is
quotient out by $\wh{B}(k)^k(X; A)$, an appropriate subgroup will be quotient out in
every other entry in (\ref{eq 5}) and the resulting diagram and the corresponding
induced maps, which will be denoted by the same symbols, coincide with (\ref{eq 1})
under the isomorphism (\ref{eq 4.5}).

To prove the induced map $R:\wh{H}(k)^k(X; A)\to\Omega^k_A(X)$ is well defined, note that
$R(\wh{B}(k)^k(X; A))=0$, which follows from that fact that if $\wh{d}(c, T)=(\delta c,
-c-\delta T, 0)\in\wh{B}(k)^k(X; A)$, then $R(\delta c, -c-\delta T, 0)=0$.

Note that $I(\wh{d}(c, T))=I(\delta c, -c-\delta T, 0)=(\delta c, -\delta c)\in B^k(X; A)
\times B^k(X; \R)$. Thus the induced map
$$I:\wh{H}(k)^k(X; A)\to\frac{Z^k(X; A)\times B^k(X; \R)}{B^k(X; A)\times B^k(X; \R)}
\cong H^k(X; A)$$
is well defined.

Let $\eta\in\Omega^{k-1}_A(X)$. Then
$$a(\eta)=(0, \eta, d\eta)=(0, \eta, 0)=\wh{d}(\eta, 0)\in\wh{B}(k)^k(X; A).$$
Thus the induced map \dis{a:\frac{\Omega^{k-1}(X)}{\Omega^{k-1}_A(X)}\to\wh{H}(k)^k(X; A)}
is well defined.

Finally, if $(u, v)\in Z(X; j_A)^{k-1}$ is a coboundary; i.e., there exists $(\wt{u}, 
\wt{v})\in C(X; j_A)^{k-2}$ such that $(u, v)=\delta_{j_A}(\wt{u}, \wt{v})=(-\delta\wt{u},
\delta\wt{v}-j_A(\wt{u}))$, then
$$i((u, v))=i(\delta_{j_A}(\wt{u}, \wt{v}))=i(-\delta\wt{u}, \delta\wt{v}-j_A(\wt{u}))=
(\delta\wt{u}, \delta\wt{v}-j_A(\wt{u}), 0)=\wh{d}(\wt{u}, \wt{v}).$$
Thus the induced map $i:H^{k-1}(X; j_A)\to\wh{H}(k)^k(X; A)$ is well defined. 

We take the quotient $Z^k(X; \R)/B^k(X; \R)$ to get $H^k(X; \R)$ in the right-most entry 
of (\ref{eq 5}). Similarly we take the quotient $\Omega^{k-1}_{d=0}(X)/\im(d)$ in the 
left-most entry in (\ref{eq 5}) to get $H_{\dr}^{k-1}(X)$. By the de Rham isomorphism we 
get $H^{k-1}(X; \R)$, and the composition of the inverse of the de Rham isomorphism with 
the two maps out of $H^{k-1}_{\dr}(X)$ in (\ref{eq 5}) are equal to the ones in (\ref{eq 1}) 
respectively.

It is easy to see that all the induced maps of (\ref{eq 5}) coincide with the maps in
(\ref{eq 1}) by the isomorphism (\ref{eq 4.5}). Thus (\ref{eq 5}) induces (\ref{eq 1}).
\end{proof}

\begin{remark}\label{remark 1}~~~~~~~~~~~
\begin{enumerate}
  \item One can see that all the maps in the hexagon (\ref{eq 5}) are natural; i.e., they 
        respect pullback by smooth maps. Thus the hexagon in Theorem \ref{thm 1} can be
        regarded as a diagram of functors and natural transformations.
  \item If $H$ is taken to be the simplicial cohomology, one must extend smooth 
        differential forms on $X$ to piecewise linear differential forms \cite{S75} to 
        get the naturality of the hexagon (\ref{eq 5}). It is motivated by \cite{D76} 
        when proving the naturality of the de Rham theorem (see also 
        \cite[Chapter 9]{GM14}). To explain the reason of such an extension, note that 
        if we take a smooth triangulation $K$ of the smooth manifold $X$ in (\ref{eq 5}) 
        and take $Z^\bullet$ to be simplicial cocycles, one loses the naturality of the 
        maps in (\ref{eq 5}) as differential forms cannot be pulled back by simplicial 
        maps. To remedy this deficiency, one has to extend to piecewise linear 
        differential form, which is a family of differential forms 
        $$\set{\omega_\sigma|\sigma \textrm{ is a closed simplex of } K}$$ 
        such that each $\omega_\sigma$ is required to be smooth on $\sigma$ and to be
        compatible when restricting to its faces. If we take a smooth triangulation $K$ 
        of $X$, the de Rham complex $\set{\Omega(X), d}$ of smooth differential forms 
        on $X$ embeds into the de Rham complex $\set{\Omega(\abs{K}), d}$ of piecewise 
        linear differential forms on $\abs{K}$. The proof of (\ref{eq 5}) remains the 
        same and we still get the naturality.
\end{enumerate}
\end{remark}

\subsection{Refined hexagon for differential $K$-theory}
In this subsection we construct a hexagon for differential $K$-theory on the cocycle
level and prove that it induces the hexagon on the cohomology level.

Instead of using the differential cocycle model \cite[Section 4.4]{HS05} of differential 
$K$-theory, we use the model of differential $K$-theory given in \cite{FL10}.

Let $X$ be a compact manifold. Recall from Example \ref{exam 2} that $\wh{K}(X):=
\wh{\vect}(X)/\sim$. Let $\wh{\vect}^{-1}_{\LL}(X)$ be the sub semi-group of 
$\wh{\vect}(X)$ whose elements $(E, h, \nabla, \phi)$ satisfy the relation $\ch(\nabla)
-\rk(E)=-d\phi$.

Define an abelian monoid
$$\vect'(X):=\set{(E, \omega)\in\vect(X)\times\Omega^{\even}_{d=0}(X)|[\omega]=j_\Q
(\ch(E))}.$$
Two elements of the form $(E, \omega)$ and $(F, \beta)$ are equal if $E\cong F$ and
$\omega=\beta\mod\Omega^{\even}_{\exact}(X)$. Define $K'(X):=K(\vect'(X))$ to be the
symmetrization of $\vect'(X)$. We write the class of $(E, \omega)\in\vect'(X)$ in $K'(X)$
as $[E, \omega]$.
\begin{lemma}\label{lemma 1}
For any compact manifold $X$, $K'(X)\cong K(X)$.
\end{lemma}
\begin{proof}
Note that every element of $K'(X)$ can be written as $[E, \omega]-[F, \beta]$.
$[E, \omega]-[F, \beta]=[E', \omega']-[F', \beta']$ if and only if there exists
$(G, \eta)\in\vect'(X)$ such that $E\oplus F'\oplus G\cong E'\oplus F\oplus G$
and $\omega-\beta=\omega'-\beta'\mod\Omega^{\even}_{\exact}(X)$.

Define a map $f:K'(X)\to K(X)$ by
$$f([E-F, \omega-\beta])=[E]-[F].$$
We prove that $f$ is a well defined group isomorphism. Suppose $[E-F, \omega-\beta]=
[E'-F', \omega'-\beta']\in K'(X)$, then there exists $(G, \eta)\in\vect'(X)$ such that
$E\oplus F'\oplus G\cong E'\oplus F\oplus G$ and $\omega-\beta=\omega'-\beta'\mod
\Omega^{\even}_{\exact}(X)$. Thus $[E]-[F]=[E']-[F']$, which implies that $f$ is well
defined. Obviously $f$ is a surjective group homomorphism. For the injectivity of $f$,
suppose $f([E, \omega]-[F, \beta])=[E]-[F]=0$. Then there exists $G\in\vect(X)$ such 
that $E\oplus G\cong F\oplus G$. Since $[\omega]-[\beta]=r(\ch(E)-\ch(F))=0$, it follows 
that $\omega=\beta\mod\Omega^{\even}_{\exact}(X)$. Thus $[E, \omega]-[F, \beta]=0\in 
K'(X)$.
\end{proof}
\begin{thm}\label{thm 2}
Let $X$ be a compact manifold. Consider the following diagram
\begin{equation}\label{eq 6}
\xymatrix{\scriptstyle & \scriptstyle \wh{\vect}^{-1}_{\LL}(X) \ar[dr]_i \ar[rr]^B &
\scriptstyle & \scriptstyle \vect'(X) \ar[dr]^{\ch'} & \scriptstyle \\ \scriptstyle
\Omega^{\odd}_{d=0}(X) \ar[ur]^\alpha \ar[dr]_\iota & \scriptstyle & \scriptstyle 
\wh{\vect}(X) \ar[ur]^{I'} \ar[dr]_R & \scriptstyle & \scriptstyle Z^{\even}(X; \R) & 
\scriptstyle \\ 
\scriptstyle & \scriptstyle \Omega^{\odd}(X) \ar[ur]^a \ar[rr]_d & \scriptstyle &
\scriptstyle \Omega^{\even}_{\BU}(X) \ar[ur]_{\der} & \scriptstyle }
\end{equation}
where all the maps will be defined in the proof. Every square and triangle commutes. The
maps $I'$ and $R$ are surjective and $I'\circ a=0$. Moreover, (\ref{eq 6}) induces
(\ref{eq 2}).
\end{thm}
\begin{proof}
The maps are defined as follows:
\begin{displaymath}
\begin{split}
I'(E, h, \nabla, \phi)&=(E, \ch(\nabla)+d\phi),\\
R(E, h, \nabla, \phi)&=\ch(\nabla)+d\phi,\\
a(\omega)&=(0, 0, d, \omega),\\
i(E, h, \nabla, \phi)&=(E, h, \nabla, \phi),\\
B(E, h, \nabla, \phi)&=(E, \ch(\nabla)+d\phi),\\
\alpha(\omega)&=(0, 0, d, \omega),\\
\ch'(E, \omega)&=j_A(\der(\omega)),
\end{split}
\end{displaymath}
and $\iota$ and $\der$ are defined in a similar way as in Theorem \ref{thm 1}. 

The map $R$ is clearly surjective. 

For the surjectivity of $I'$, let $[E, \omega]\in K'(X)$. Since $\omega=\ch(\nabla)+
d\phi$ for some connection $\nabla$ on $E\to X$ and for some $\phi\in\Omega^{\odd}(X)$, 
it follows that $I'(E, h, \nabla, \phi)=[E, \ch(\nabla)+d\phi]=[E, \omega]$.

From the definitions, we immediately have $R\circ a=d$ and $I'\circ i=-B$. To prove the 
commutativity of the right hand square, let $(E, h, \nabla, \phi)\in\wh{\vect}(X)$. Then
\begin{equation}\label{eq 7.5}
\begin{split}
\ch'(I'(E, h, \nabla, \phi))&=\ch'(E, \ch(\nabla)+d\phi)=\int(\ch(\nabla)+d\phi)\\
&=\der(\ch(\nabla)+d\phi)=\der(R(E, h, \nabla, \phi)).
\end{split}
\end{equation}
For the commutativity of the left hand square, let $u\in Z^{\odd}(X; \R)$. Then
$$i(\alpha(\omega))=i(0, 0, d, \omega)=(0, 0, d, \omega)=a(\omega)=a(\iota(\omega)).$$
We prove that (\ref{eq 6}) induces (\ref{eq 2}). As in the proof of Theorem \ref{thm 1},
we show that once we quotient out $\wh{\vect}(X)$ by the equivalence relation $\sim$, every
other entry in (\ref{eq 8}) is quotient out by a suitable subgroup so that the induced hexagon
coincides with (\ref{eq 2}), i.e., all the entries and the corresponding maps are equal.

Let $\E_i=(E^i, h^i, \nabla^i, \phi^i)\in\wh{\vect}(X)$, where $i=1, 2$, be such that
$\E_1-\E_2=0\in\wh{K}_{\FL}(X)$, i.e., there exists $(F, h^F, \nabla^F, \phi^F)\in
\wh{\vect}(X)$ such that $E^1\oplus F\cong E^2\oplus F$ and
\begin{equation}\label{eq 8}
\phi^1-\phi^2=\CS(\nabla^2\oplus\nabla^F, \nabla^1\oplus\nabla^F)\mod\Omega^{\odd}_{\exact}
(X).
\end{equation}
Then
\begin{displaymath}
\begin{split}
R(\E_1-\E_2)&=\ch(\nabla^1)+d\phi^1-\ch(\nabla^2)-d\phi^2\\
&=d(\CS(\nabla^1\oplus\nabla^F, \nabla^2\oplus\nabla^F)+\phi^1-\phi^2)=0\textrm{ from
(\ref{eq 8})}.
\end{split}
\end{displaymath}
Thus the induced map $R:\wh{K}_{\FL}(X)\to\Omega^{\even}_{\BU}(X)$ is well defined.

The induced map $\ch':K'(X)\to H^{\even}(X; \R)$ is well defined since for $[E-F,
\omega-\beta]=0\in K'(X)$ and for any $z\in Z_{\even}(X)$, we have
$$\ch'([E, \omega]-[F, \beta])(z)=\ch'(E, \omega)(z)-\ch'(F, \beta)(z)=\int_z\omega-\beta
=\int_zd\alpha=0,$$
where $\alpha\in\Omega^{\odd}(X)$.

We claim that the induced map $I':\wh{K}_{\FL}(X)\to K'(X)$ is well defined. To see this,
suppose $\E=\F\in\wh{K}_{\FL}(X)$. Then there exists $(G, h^G, \nabla^G, \phi^G)\in
\wh{\vect}(X)$ such that $E\oplus G\cong F\oplus G$ and $\phi^F-\phi^E=\CS(\nabla^E\oplus
\nabla^G, \nabla^F\oplus\nabla^G)$, which implies that $d(\phi^F-\phi^E)=\ch(\nabla^E)-
\ch(\nabla^F)$. Then
$$I'(\E-\F)=[E, \ch(\nabla^E)+d\phi^E]-[F, \ch(\nabla^F)+d\phi^F]=0\in K'(X).$$
Define the map $I:\wh{K}_{\FL}(X)\to K(X)$ to be the composition
\cdd{\wh{K}_{\FL}(X) @>I'>> K'(X) @>f>\cong> K(X),}
where $f$ is given in Lemma \ref{lemma 1}. Thus
$$I(\E-\F)=f(I'(\E-\F))=f([(E-F, \omega-\beta)])=[E]-[F],$$
which coincides with the map $I:\wh{K}_{\FL}(X)\to K(X)$ in (\ref{eq 2}).

The induced map $\der:\Omega^{\even}_{\BU}(X)\to H^{\even}(X; \R)$ is defined to be the
composition
\cdd{\Omega^{\even}_{\BU}(X) @>\der>> Z^{\even}(X; \R) @>[\cdot]>> H^{\even}(X; \R)}
which coincide with the map $\der$ in (\ref{eq 2}). The induced map
$\alpha:H^{\odd}_{\dr}(X)\to K^{-1}_{\LL}(X; \R/\Z)$ is well defined and coincides with
the map $\alpha$ in (\ref{eq 2}) by the definition of $K^{-1}_{\LL}(X; \R/\Z)$. Since
$i:\wh{\vect}^{-1}_{\LL}(X)\to\wh{\vect}(X)$ is the canonical inclusion map, its induced
map $i:K^{-1}_{\LL}(X; \R/\Z)\to\wh{K}_{\FL}(X)$ is well defined and coincides with the 
map $i$ in (\ref{eq 2}).

We prove that the induced map \dis{a:\frac{\Omega^{\odd}(X)}{\Omega^{\odd}_{\BU}(X)}\to
\wh{K}_{\FL}(X)} is well defined. Let $\omega\in\Omega^{\odd}_{\BU}(X)$. Then $\omega=
\ch^{\odd}([g])+d\alpha$, where $g:X\to\U(n)$ is a smooth map for some $n\in\N$, $\alpha
\in\Omega^{\even}(X)$ and $\ch^{\odd}([g]):=\CS(g^{-1}dg, d)$. Here we have identified
$K^{-1}(X)\cong[X, \U]$. Since $a(\omega)=a(\ch^{\odd}([g]))+a(d\alpha)$ and $a(d\alpha)=
(0, 0, d, d\alpha)=0$, it suffices to show that $a(\ch^{\odd}([g]))=0$. Since $R\circ a=d$,
it follows from the exactness of the following sequence
\cdd{0 @>>> K^{-1}_{\LL}(X; \R/\Z) @>i>> \wh{K}_{\FL}(X) @>R>> \Omega^{\even}_{\BU}(X) 
@>>> 0}
that $a(\ch^{\odd}([g]))\in K^{-1}_{\LL}(X; \R/\Z)$. Consider the following sequence
$$\xymatrix{ & & K^{-1}_{\LL}(X; \R/\Z) \ar[dr]^i & \\ K^{-1}(X) \ar[r]^{\ch^{\odd}} &
H^{\odd}_{\dr}(X) \ar[ur]^\alpha \ar[dr]_\iota & & \wh{K}_{\FL}(X) \\ & & 
\frac{\Omega^{\odd}(X)}{\Omega^{\odd}_{\BU}(X)} \ar[ur]^a & }$$
Since the upper sequence is exact \cite[\S 7.21]{Kar87} and the square commutes, it
follows that
$$a(\ch^{\odd}([g]))=a(\iota(\ch^{\odd}([g])))=i(\alpha(\ch^{\odd}([g])))=i(0)=0.$$
Thus the map \dis{a:\frac{\Omega^{\odd}(X)}{\Omega^{\odd}_{\BU}(X)}\to\wh{K}_{\FL}(X)} 
is well defined.

Define $\ch:K(X)\to H^{\even}(X; \R)$ to be the composition
\cdd{K(X) @>f^{-1}>> K'(X) @>\ch'>> H^{\even}(X; \R)}
Note that this map $\ch:K(X)\to H^{\even}(X; \R)$ coincides with the map
$\ch:K(X)\to H^{\even}(X; \R)$ defined in (\ref{eq 2}).

The commutativity of the following square
\cdd{\wh{K}_{\FL}(X) @>I>> K(X) \\ @VRVV @VV\ch V \\ \Omega^{\even}_{\BU}(X) @>>\der> 
H^{\even}(X; \R)}
follows from the fact that $\ch\circ I=\ch'\circ f^{-1}\circ f\circ I'=\ch'\circ I'=
\der\circ R$, where the last equality follows from (\ref{eq 7.5}). Thus (\ref{eq 6})
induces (\ref{eq 2}).
\end{proof}
\begin{remark}
In Theorem \ref{thm 1} we have the exact sequence
\cdd{0 @>>> Z(X; j_A)^{k-1} @>i>> \wh{Z}(k)^k(X; \R/A) @>R>> \Omega^k_A(X) @>>> 0}
while in Theorem \ref{thm 2} the ``corresponding" sequence
\cdd{0 @>>> \wh{\vect}^{-1}_{\LL}(X) @>i>> \wh{\vect}(X) @>R>> \Omega^{\even}_{\BU}(X)
@>>> 0}
is not exact and not even $R\circ i=0$. However, when these entries are quotient out by
the coboundaries we have the exactness and $R\circ i=0$. This is due to the fact that
$\wh{\vect}_{\LL}^{-1}(X)$ is a sub semi-group of $\wh{\vect}(X)$ but not a sub-monoid
and, after symmetrizing, $K^{-1}_{\LL}(X; \R/\Z)$ is a \emph{proper} subgroup of
$K(\wh{\vect}_{\LL}^{-1}(X))$ consisting of elements with virtual rank zero.
\end{remark}
\bibliographystyle{amsplain}
\bibliography{MBib}
\end{document}